\documentclass[12pt,a4paper]{amsart}

\usepackage[normalem]{ulem}
\usepackage{amsmath}
\usepackage{amsfonts}
\usepackage{amssymb}
\usepackage{amsthm}
\usepackage{graphicx}
\usepackage{microtype}
\usepackage[usenames,dvipsnames]{xcolor}
\usepackage{pgf,tikz}
\usepackage{mathrsfs}
\usetikzlibrary{arrows}
\usepackage{caption}
\usepackage{soul}
\usepackage[pdftex,plainpages=false,bookmarks=true,breaklinks=true,linkcolor=brown,citecolor=brown,colorlinks]{hyperref}
\usepackage{bm}
\usepackage{cancel} 
\usepackage{combelow}

\newtheorem{theorem}{Theorem}[section]
    
\newtheorem{lemma}[theorem]{Lemma}

\theoremstyle{definition}
   
\newtheorem{example}[theorem]{Example}

\author{Eiichi Matsuhashi}

\address{Department of Mathematics, Shimane University, Matsue, Shimane 690-8504, Japan.}
\email{matsuhashi@riko.shimane-u.ac.jp}

\author{Yoshiyuki Oshima}
\address{Department of Mathematics, Shimane University, Matsue, Shimane 690-8504, Japan.}
\email{n20d101@matsu.shimane-u.ac.jp}

\title{Some theorems on colocally connected continua }
\begin{document}


\maketitle 

\begin{abstract}
We show that each refinable map  preserves colocal  connectedness of the domain while a proximately refinable map does not necessarily.   Also, we prove that  colocal connectedness  is a  Whitney property and is not a Whitney reversible property.  
    
\end{abstract}


\section{Introduction}
In this paper, unless otherwise stated, all spaces are assumed to be metrizable. When we use the term $function$, we do not assume it to be continuous necessarily while we require $map$ to be continuous. 
Let $X$ be a continuum and $x \in X$. We say that  $X$ is {\it colocally connected} if for each $x \in X$ and  each neighborhood $V \subseteq X$ of $x$, there exists an open neighborhood $U \subseteq X$ of $x$ such that $x \in U \subseteq V$ and  $X \setminus U$ is connected. A continuum $X$ is said be $aposyndetic$ if for any two distinct points $x,y \in X$, there exists a subcontinuum $T \subseteq X$ such that $x \in {\rm Int}_X T \subseteq T \subseteq X \setminus \{y\}$, where ${\rm Int}_X T$ denotes the interior of $T$ in $X$. Jones introduced aposyndetic continua in \cite{jonesap}. Since then, aposyndetic continua have been studied for many years. 
It is known that every colocally connected continuum is aposyndetic \cite[Remark 5.4.15]{macias}. 
As the consequence,  colocal connectedness implies many properties of continua (see \cite[Figure 6]{D}. See also \cite[p.239]{bobok} for other properties derived from  colocal connectedness).

If $(X,d)$ is a space and $A \subseteq X$, then we denote ${\rm sup}\{d(a,b) \ | \ a,b \in A \}$ by ${\rm diam}_d A$. Let $(X,d_X)$ and $(Y,d_Y)$ be continua and let $\varepsilon > 0$. A surjective map  $f: X \to Y$ is called an $\varepsilon$-$map$ if for each  $y \in Y$, ${\rm diam}_{d_X} f^{-1}(y) < \varepsilon$. If $g, ~ g_{\varepsilon} : X \to Y$ are  surjective maps such that $g_{\varepsilon}$ is an $\varepsilon$-map and $d_Y(g(x), g_{\varepsilon}(x)) < \varepsilon$ for each $x \in X$, then $g_{\varepsilon}$ is called an $\varepsilon$-$refinement$ $of$ $g$.  A surjective map $r : X \to Y$ is called a $refinable ~ map$ if for each $\varepsilon > 0$, there exists an $\varepsilon$-refinement $r_{\varepsilon} : X \to Y$ of $r$.   The notion of a refinable map was introduced in \cite{ford}. 
 
 If $X$ is 
a continuum, then we  denote the  space of all nonempty subcontinua of $X$ endowed with the Hausdorff metric by $C(X)$. $C(X)$ is called the $hyperspace$ $of$ $X$. A $Whitney$
$map$  is a map $\mu : C(X) \to [0,\mu(X)]$ satisfying  $\mu(\{x\})=0$ for each $x \in X$ and $\mu(A) < \mu(B)$  whenever $A, B \in C(X)$ and $A \subsetneq B$. It is well-known  that for each Whitney map  $\mu: C(X) \to [0,\mu(X)]$   
and  each $t \in [0, \mu(X)]$,  $\mu^{-1}(t)$ is a continuum (\cite[Theorem 19.9]{illanes}). 
 
 A topological property $P$ is called a {\it Whitney property} if a continuum $X$ has property $P$, so does $\mu^{-1}(t)$ for each Whitney map  $\mu$ for $C(X)$ and  each $t \in [0, \mu(X))$.   Also, a topological property $P$ is called a {\it Whitney reversible property} provided that whenever $X$ is a continuum such that $\mu^{-1}(t)$ has property $P$ for each Whitney map $\mu$ for $C(X)$ and {for} each $t \in (0, \mu(X))$, then $X$ has property $P$. Many researchers have studied these properties so far.  As for information about Whitney properties and Whitney reversible properties, for example, see \cite[Chapter 8]{illanes}.

   In \cite{hosokawa}, Hosokawa proved that each refinable map defined on a continuum preserves aposyndesis. Also, in \cite{petrus} Petrus proved  that aposyndesis is a   Whitney property and is not a Whitney reversible property.  In this paper, we show that   each refinable map  preserves colocal  connectedness of the domain while a proximately refinable map does not necessarily (for the definition of a proximately refinable map, see section 2).   Also, we prove that  colocal connectedness  is a  Whitney property and is not a Whitney reversible property.  
 

\section{refinable maps and proximately refinable maps defined on colocally connected continua}

In this section, we deal with topics on refinable maps. First, we show an
example of a refinable map between colocally connected continua which is not a homeomorphism. 

\begin{example}
Take a continuum $X$  and a refinable map $f  : X \to f(X)$ such that $X$ and $f(X)$ are not homeomorphic (for an example of a refinable map which is not a homeomorphism, see \cite{ford}). Let $Y$ be a nondegenerate continuum. By \cite[Lemma 3.5]{loncar3}, $X \times Y$ and $f(X) \times Y$ are colocally connected. Then, a map $F : X \times Y \to f(X) \times Y ; (x,y) \mapsto (f(x),y)$ is a refinable map which is not a homeomorphism.

\end{example}




The following theorem is the main result in this section. 
Before the theorem, we give a notation. If $X$ is a space and $A \subset X$, then ${\rm Cl}_X A$ denotes the closure of $A$ in $X$.

\begin{theorem}
Let $(X,d_X)$ and $(Y,d_Y)$ be continua and let $f : X \to Y$ be a refinable map. If $X$ is colocally connected, then so is $Y$. 
\label{refinableco}
\end{theorem}

\begin{proof}
Let $y \in Y$ and let $V \subseteq Y$ be an open neighborhood of $y$. Since  $f$ is a refinable map, there exists a sequence $\{f_{\frac{1}{n}}\}_{n=1}^{\infty}$ of $\frac{1}{n}$-refinements of $f$. We may assume that $\lim f_{\frac{1}{n}}^{-1}(y)$ exists. Let $\lim f_{\frac{1}{n}}^{-1}(y)=\{x\}$.   Then, it is easy to see that $x \in f^{-1}(y)$. Since $X$ is colocally connected, there exists an open neighborhood $B \subseteq X$ of $x$ such that $x \in B \subseteq {\rm Cl}_X B \subseteq f^{-1}(V)$ and $X \setminus B$ is connected. Let $\delta = {\rm inf} \{d_Y(b,c) \ | \ b \in f({\rm Cl}_X B), \ c \in Y \setminus V \}$. Then, there exists $n_0 \in \mathbb{N}$ such that $f_{\frac{1}{n_0}}^{-1}(y) \subseteq B$ and $d_Y(f(x), f_{\frac{1}{n_0}}(x)) < \frac{\delta}{2}$ for each $x \in X$.  Let $U=Y \setminus f_{\frac{1}{n_0}}(X \setminus B)$. Then, we see that $Y \setminus U=f_{\frac{1}{n_0}}(X \setminus B)$ is connected and $y \in U \subseteq f_{\frac{1}{n_0}}( {\rm Cl}_X B) \subseteq V$. Therefore, we see that $Y$ is colocally connected.  \end{proof}

A continuum is said to be $semilocally ~ connected$ if for each $x \in X$ and each neighborhood $V \subseteq X$ of $x$ there exists an open neighborhood $U \subseteq X$ of $x$ such that $x \in U \subseteq V$ and $X \setminus U$ has finitely many components. 

\begin{theorem}
Let $X$ and $Y$ be continua and let $f : X \to Y$ be a refinable map. If $X$ is semilocally connected, then so is $Y$. \label{refinablesemi}
\end{theorem}

\begin{proof}
The proof is similar to the proof of the previous result. In fact, change the proof of Theorem \ref{refinableco} as follows:

\smallskip

$\bullet$ In the proof of Theorem \ref{refinableco}, we took an open neighborhood $B \subseteq X$ of $x$ such that $X \setminus B$ is connected. Instead of it, take an open neighborhood  $B \subseteq X$ of $x$ such that $X \setminus B$ has finitely many components.

\smallskip

$\bullet$ Also, in the same proof, we stated that $Y \setminus U$ is connected. However, by taking $B$ as  above, we can easily see that $Y \setminus U$ has finitely many components.

\medskip

By  these two remarks, we can prove Theorem \ref{refinablesemi}. \end{proof}

It is known that a continuum $X$ is aposydetic if and only if $X$ is semilocally connected (see \cite{jonesap}). Hence,  Theorem \ref{refinablesemi}  implies the  following result proven by Hosokawa \cite{hosokawa} and vice versa. 

\begin{theorem}{\rm (\cite[Corollary (p.368)]{hosokawa})}
Let $X$ and $Y$ be continua and let $f : X \to Y$ be a refinable map. If $X$ is aposyndetic, then so is $Y$. 
\end{theorem}

If $(X,d)$ is a space,  $x \in X$ and $\varepsilon > 0$, then we denote the set $\{ y \in X \  | \ d(x,y) < \varepsilon\}$ by $U_d(x,\varepsilon)$. 
A function $f: (X,d_X) \to (Y,d_Y)$ between spaces is said to be $\varepsilon$-$continuous$ if for each
$x \in X$ there exists an open neighborhood $U \subset X$ of $x$  such that $f(U) \subseteq U_{d_Y}(f(x), \varepsilon)$. A surjective function  $g:(X,d_X) \to (Y,d_Y)$  is called a {\it strong $\varepsilon$-function} if for each $y \in Y$, there exists an open neighborhood $V \subseteq Y$ of $y$ such that ${\rm diam}_{d_X} g^{-1}(V) < \varepsilon$.

Let $f:(X,d_X) \to (Y,d_Y)$ be a surjective function between continua. 
A surjective
function $g: X \to Y$ is called a {\it proximate $\varepsilon$-refinement of $f$}  
if $g$ is $\varepsilon$-continuous, $g$ is a strong $\varepsilon$-function 
and $d_Y(f(x), g(x)) < \varepsilon$ for each $x \in X$. 
A surjective function $p: X \to Y$ is said to be  {\it proximately refinable} 
if for any $\varepsilon > 0$ there exists a proximate $\varepsilon$-refinement of $p$. We can easily see that every proximately refinable function is a map. The notion of a proximately refinable map was introduced in \cite{grace}.

\smallskip

The following result was proven by Grace and Vought in \cite{grace2}.

\begin{theorem}{\rm (\cite[Theorem 2]{grace2})}
A surjective map defined on a graph is proximately refinable if and only if it is monotone. 
\label{proximate}

\end{theorem}

By the following example, we see that a proximately refinable map does not preserve colocal connectedness of the domain necessarily. Before the example, we give a notation. If $X$ is a space and $A \subset X$, then ${\rm Bd}_X A$ denotes the boundary of $A$ in $X$.

\begin{example}
Let $X$ be the subcontinuum in $\mathbb{R}^2$ defined by $X = {\rm Bd}_{\mathbb{R}^2} [0,1]^2 \cup \{(\frac{1}{2},y) \  | \  0 \le y \le 1\}$.   Let $Y$ be the quotient space obtained from $X$ by shrinking $\{(\frac{1}{2},y) \  | \  0 \le y \le 1\} $ to the
point and let $f : X \to Y$ be the quotient map. Then, it is easy to see that $f$ is a surjective monotone map, $X$ is a colocally connected graph, and $Y$ is not colocally connected. By Theorem \ref{proximate}, $f$ is proximately refinable. Hence, colocal connectedness of the domain is not necessarily  preserved by a proximately  refinable map.

\end{example}

\section{A Whitney property and a Whitney reversible property}

Let $X$ be a continuum and $p \in X$. We say that $p$ is a $weak$ $cut$ $point$ $of$ $X$ if there exist distinct points $x,y \in X \setminus \{p\}$ such that for each  subcontinuum $C \subseteq X$ with $x,y \in C$, $p \in C$. If a point of $X$ is not a weak cut point, then the point is called a  $non$-$weak$ $cut$ $point$. It is easy to see that if a continuum $X$ is colocally connected, then $X$ is a continuum having only non-weak cut points.

As mentioned earlier, aposyndesis is a Whitney property and is not a Whitney reversible property. Also, it is known that the property of having only non-weak cut points is a Whitney property and is not a Whitney reversible property (\cite[Theorem 2.3 and Theorem 2.10]{bautista}).  In this section, we prove that  colocal connectedness  is a  Whitney property and is not a Whitney reversible property. 

The main aim of this section is to prove Theorem \ref{whitneyproperty}. To prove the theorem, we need  Lemma \ref{exe}. Before stating those results, we give notation. If $(X,d)$ is a continuum and $A$ is a closed subset of $X$, then we denote the set $\{x \in X \ | \ {\rm there \ exists }\ a \in A \ {\rm such \ that } \ d(x,a) < \varepsilon \}$ by $N_d(A, \varepsilon)$. Also, $H_d$ denotes the Hausdorff metric on $C(X)$ induced  by $d$.  Finally, if $\mathcal{A}$ is a family of subsets of $X$, then we denote ${\rm sup} \{{\rm diam}_d A \ | \ A \in \mathcal{A}\}$ by mesh$_d\mathcal{A}$.

\begin{lemma}{\rm(\cite[Exercise 4.33 (b)]{nadler1})}
Let $(X,d)$ be a continuum and let $\mu : C(X) \to [0, \mu(X)]$ be a Whitney map. Then, for each $\varepsilon > 0$, there exists $\delta > 0$ such that if $A,B \in C(X)$ satisfy $B \subseteq N_d(A, \delta)$ and $|\mu(A) - \mu(B)| < \delta$, then $H_d(A,B) < \varepsilon$. 
\label{exe}
\end{lemma}

\begin{theorem}
Colocal connectedness is a Whitney property.
\label{whitneyproperty}
\end{theorem}

\begin{proof}

Let $(X,d)$ be a colocally connected continuum, let  $\mu: C(X) \to [0,\mu(X)]$ be a Whitney map and  let $t \in (0,\mu(X))$. Let $A \in \mu^{-1}(t)$ and let $\varepsilon > 0$.   We will find an  open neighborhood $\mathcal{O} \subseteq \mu^{-1}(t)$ of $A$ such that $\mathcal{O} \subseteq U_{H_d}(A, \varepsilon)$ and $\mu^{-1}(t) \setminus \mathcal{O}$ is connected.

Take $\delta > 0$ as in the statement of Lemma \ref{exe}. Since $X$  is colocally connected, for each $a \in A$, there exists an open neighborhood $O_a \subset X$ such that ${\rm diam}_d O_a < \delta$, $X \setminus O_a$ is connected and $\mu(X \setminus O_a) > t$. Then, $\{O_a\}_{a \in A}$ covers $A$. Also, since $X$ is colocally connected and $A$ is compact, we can find a finite collection $\mathcal{H}=\{H_i\}_{i=1}^n$ of open subsets of  $X$ such that 

\smallskip

(1) $\mathcal{H}$ covers $A$, 

(2) for each $i=1,2,\ldots,n$, $A \cap H_i \neq \emptyset$,

(3) for each $i=1,2,\ldots,n$, $X \setminus H_i$ is connected, and

(4) if $H, H' \in \mathcal{H}$ and $H \cap H' \neq \emptyset$, then there exists $a \in A$ such that $H \cup H' \subseteq O_a$. 

\smallskip

We may assume that   if $H_i, H_j \in \mathcal{H}$  and  $|i-j| \le 1$, then $H_i \cap H_j \neq \emptyset$. In this case, note that it is possible  $H_i=H_j$ despite $i \neq j$. For each $i=1,2,\ldots,n$, let $\mathcal{H}_i = \{C \in \mu^{-1}(t) \ |\  C \subseteq X \setminus H_i\}$. Now, we show the following:

\smallskip

\begin{itemize}
    \item[$\bullet$] $\bigcup_{i=1}^n \mathcal{H}_i$ is a subcontinuum of $\mu^{-1}(t)$.
\end{itemize}

 \smallskip
 
 First, by \cite[Exercise 27.7]{illanes}, each $\mathcal{H}_i$ is a subcontinuum of $\mu^{-1}(t)$ (note that each $\mathcal{H}_i$ is nonempty by the fact that  $\mu(X \setminus O_a)>t$ for each $a \in A$ and \cite[Theorem 14.6]{illanes}). Hence, it is enough to show that $\mathcal{H}_i \cap \mathcal{H}_{i+1} \neq \emptyset$ for each $i=1,2,\ldots,n-1$. Let $i=1,2.\ldots,n-1$. Since $H_i \cap H_{i+1} \neq \emptyset$, there exists $a \in A$ such that $H_i \cup H_{i+1} \subseteq O_a$. Since $\mu(X \setminus O_a)>t,$ by \cite[Theorem 14.6]{illanes} we see that there exists a subcontinuum $K \subseteq X \setminus O_a$ such that $\mu(K)=t$. Note that $K \in \mathcal{H}_i \cap \mathcal{H}_{i+1}$. Therefore, we see that $\bigcup_{i=1}^n \mathcal{H}_i$ is a continuum.

 Let $\mathcal{O}= \mu^{-1}(t) \setminus \bigcup_{i=1}^n \mathcal{H}_i$. By the above argument, $\mu^{-1}(t) \setminus \mathcal{O} = \bigcup_{i=1}^n \mathcal{H}_i$ is connected. Also, by (2), it follows that $A \notin \bigcup_{i=1}^n \mathcal{H}_i$. Therefore, $A \in \mathcal{O}$. Thus, $\mathcal{O} \subseteq \mu^{-1}(t)$ is an open neighborhood of $A$.  Finally, we show that  $ \mathcal{O} \subseteq U_{H_d}(A, \varepsilon)$. Let $B \in \mathcal{O}$. Then, $B \notin \bigcup_{i=1}^n \mathcal{H}_i$. Hence, for each $i=1,2,...,n$, $B \cap H_i \neq \emptyset$. Since $\mathcal{H}$ covers $A$ and ${\rm mesh}_d \mathcal{H} < \delta$, it follows that $A \subseteq N_d(B,\delta)$. Note that $\mu(A)=\mu(B)=t$. Hence, by Lemma \ref{exe}, we see that $H_d(A,B) < \varepsilon$. Therefore, $\mathcal{O} \subseteq U_{H_d}(A, \varepsilon)$. 
 
 Thus, we see that $\mu^{-1}(t)$ is colocally connected. This completes the proof. \end{proof}
 
 The following example is appeared in \cite[Example 2.10]{bautista} to show that the property of having only non-weak cut points is not a Whitney reversible property. By the same example, we see that colocal connectedness is not a Whitney reversible property. 

\begin{example}{\rm (cf. \cite[Example 2.10]{bautista})}
Let $X$ be a dendrite whose  branch points are dense in $X$ and let  $\mu : C(X) \to [0,\mu(X)]$ be a Whitney map. By \cite[Theorem 4.8]{good}, for each $t \in (0, \mu(X))$, $\mu^{-1}(t)$ is homeomorphic to the Hilbert cube. Hence, $\mu^{-1}(t)$ is colocally connected for each $t \in (0, \mu(X))$. Since $X$ is a dendrite, $X$ is not colocally connected. Therefore, we see that colocal connectedness is not a Whitney reversible property. 
\end{example}

\end{document}